\documentclass[12pt]{article}
\usepackage{graphics,amsmath,amssymb}
\usepackage{amsthm}
\usepackage{amsfonts}
\usepackage{latexsym}
\usepackage{amscd}
\usepackage{atbegshi}
\AtBeginDocument{\AtBeginShipoutNext{\AtBeginShipoutDiscard}}

\newtheorem{theorem}{Theorem}[section]
\newtheorem{corollary}{Corollary}[section]

\newtheorem{conjecture}{Conjecture}[section]

\begin{document}
\begin{center}
\title{Two Upper Bounds for both the van der Waerden Numbers $W(r, k + 1)$ and $W(r + 1, k)$, that show the existence of a recurrence relation}
\author{\textbf{Robert J. Betts}}
\maketitle
\emph{The Open University\\Postgraduate Department of Mathematics and Statistics~\footnote{2012--2013. Studies postponed due to serious physical illness.}\\ (Main Campus) Walton Hall, Milton Keynes, MK7 6AA, UK\\
Robert\_Betts@alum.umb.edu}
\end{center}
\begin{abstract}
Using a method we have utilized previously, namely through a finite power series expansion which also sometimes is known as the ``radix polynomial" representation of an integer, we find an upper bound for a van der Waerden number that has a recurrence property.,,
\end{abstract}
\section{Introduction}
Any van der Waerden number $W(r, k)$~\cite{Graham and Rothschild, Graham and Spencer, van der Waerden, Landman and Robertson, Khinchin}, where $r$ is the number of integer colorings in the interval $[1, W(r, k)]$ on $\mathbb{R}$, and $k$ is the number of terms in the arithmetic progression within the interval, can be expanded into an integer polynomial in $k$, meaning as
\begin{equation}
W(r, k) = c_{m}k^{m} + c_{m - 1}k^{m - 1} + \cdots + c_{0} \in [k^{m}, k^{m + 1}) \subset \mathbb{R},
\end{equation}
where \(c_{m} \in [1, k - 1]\), \(c_{m - 1}, c_{m - 2}, \ldots, c_{0} \in [0, k - 1]\) are all integers and where each such van der Waerden number can be found within some interval $[k^{m}, k^{m + 1})$ on $\mathbb{R}$ for some positive integer exponent $m$. Ordinarily,
\begin{equation}
(c_{m}c_{m - 1}\cdots c_{0})_{k},
\end{equation}
would denote the base $k$ expansion~\cite{Abramowitz, Rosen}, of the integer $W(r, k)$. Nevertheless the expansion we denote as 
\begin{equation}
c_{m}k^{m} + c_{m - 1}k^{m - 1} + \cdots + c_{0},
\end{equation}
as a finite power series expansion in powers of $k$ or as an integer polynomial representation in $k$ for $W(r, k)$, still is an integer we can sum to the base ten integer $W(r, k)$~\cite{Betts3, Betts2, Betts1}. In general we have always a positive real exponent \(\log_{k}W(r, k) = \delta_{k}(r, k)\) such that \(k^{m} \leq k^{\delta_{k}(r, k)} < k^{m + 1}\), where \(W(r, k) = k^{\delta_{k}(r, k)}\) and where
\begin{equation}
m = \lfloor \delta_{k}(r, k) \rfloor.
\end{equation}
The reader should understand the inequality \(k^{m} \leq W(r, k) < k^{m + 1}\) \emph{always is true} since
$$
c_{m}k^{m} + c_{m - 1}k^{m - 1} + \cdots + c_{0} \in [k^{m}, k^{m + 1}) \Longrightarrow W(r, k) \in [k^{m}, k^{m + 1})
$$ 
is true, because 
$$
W(r, k) = c_{m}k^{m} + c_{m - 1}k^{m - 1} + \cdots + c_{0}.
$$
All integers inclusive from ten to $99$ lie within the same interval $[10, 10^{2})$. All integers from four to seven lie within the same interval $[2^{2}, 2^{3})$. Similarly all integers inclusive from $k^{m}$ to $k^{m + 1} - 1$, which as we have seen does include $W(r, k)$ as well, lie within the interval $[k^{m}, k^{m + 1})$. In fact many applied computer scientists and in particular computer network engineers among them who are familiar with the gory details of IANA based IPv4/IPv6 addressing and network routing for packet switching applications, domain name addressing, etc.,~\cite{Kurose, Bates}, understand that all these depend heavily on base two and base sixteen number representations, so that IPv4 addresses can be assigned for instance, from 0.0.0.0 to 255.255.255.255, where \(2^{7} = 128, 2^{8} - 1 = 255, 2^{8} = 256\).\\
\indent Some researchers and authors~\cite{Kornerup}, have referred to the finite power series expansion such as the one in Eqn. (3) as being a ``radix polynomial" or representation. For our purposes we shall use the ``radix polynomial" or finite expansion of $W(r, k)$ into powers of $k$ (one may choose whatever terminology one prefers here, whether radix polynomial or finite power series expansion in $k$) only as an expansion of $W(r, k)$ such that the result is a base ten integer although the base $k$ integer (which does \emph{not} concern us here) would be denoted ordinarily as $(c_{m}c_{m - 1}\cdots c_{0})_{k}$. For example van der Waerden number \(W(2, 6) = 1132\) in base six would be denoted as $5124_{6}$ where \(m = 3, c_{3} = 5\), \(c_{2} = 1, c_{1} = 2\), \(c_{0} = 4\), while here we only are concerned with the expansion of this integer $1132$ (See Eqns. (6)--(7)) as a base ten integer but expanded into powers of six as in Eqns. (6)--(7). To illustrate what is meant by the last sentence, we are using the fact that in ordinary base ten arithmetic, both of the two different representations with the first expressed as a sum of powers of ten with the second expressed as a sum of powers of six, in
\begin{eqnarray}
1\cdot 10^{3} + 1\cdot 10^{2} + 3\cdot 10^{1} + 2&=&5\cdot 6^{3} + 1\cdot 6^{2} + 2\cdot 6^{1} + 4\nonumber\\
                                               &=&1132 \in [6^{3}, 6^{4}) \subset \mathbb{R},\nonumber
\end{eqnarray}
\emph{sum when we do arithmetic in base ten to the very same van der Waerden number} $1132$, although one of these representations for $1132$ is expressed as a sum of powers of six while the other representation is expressed as a sum of powers of ten.\\
\indent So for the seven van der Waerden numbers $W(2, 3)$, $W(2, 4)$, $W(2, 5)$, $W(2, 6)$, $W(3, 3)$, $W(3, 4)$, $W(4, 3)$, we get
\begin{eqnarray}
W(2, 3)&=&9 = 3^{2} \in [3^{2}, 3^{3}),\\
W(2, 4)&=&35 = 2\cdot 4^{2} + 3 \in [4^{2}, 4^{3}),\nonumber\\
W(2, 5)&=&178 = 5^{3} + 2\cdot 5^{2} + 3 \in [5^{3}, 5^{4}),\\
W(2, 6)&=&1132 = 5\cdot 6^{3} + 1\cdot 6^{2} + 2\cdot 6^{1} + 4 \in [6^{3}, 6^{4}),\nonumber\\
W(3, 3)&=&27 = 3^{3} \in [3^{3}, 3^{4}),\\
W(3, 4)&=&293 = 4^{4} + 2\cdot 4^{2} + 4^{1} + 1 \in [4^{4}, 4^{5}),\\ 
W(4, 3)&=&76 = 2\cdot 3^{3} + 2\cdot 3^{2} + 3^{1} + 1 \in [3^{3}, 3^{4}).
\end{eqnarray} 
\section{The Rational numbers $\frac{W(r, k + 1)}{W(r, k)}$}
The following six known van der Waerden numbers $W(r, k)$
\begin{eqnarray}
W(2, 3)&=&9, \: W(2, 4) = 35,\\
W(2, 5)&=&178, \: W(2, 6) = 1132,\\
W(3, 3)&=&27, \: W(3, 4) = 293,
\end{eqnarray}
have an interesting property. Notice the following rational numbers, namely,
\begin{eqnarray}
\frac{W(2, 4)}{W(2, 3)}&=&\frac{35}{9} = 3.88 \approx 4,\\
\frac{W(2, 5)}{W(2, 4)}&=&\frac{178}{35} = 5.08 \approx 5\nonumber\\
\frac{W(2, 6)}{W(2, 5)}&=&\frac{1132}{178} = 6.35 \approx 6,\\
\frac{W(3, 4)}{W(3, 3)}&=&\frac{293}{27} = 10.85.
\end{eqnarray}
One can see we get the approximations \(W(2, 4) \approx 4 \cdot W(2, 3)\), \(W(2, 5) \approx 5 \cdot W(2, 4)\), \(W(2, 6) \approx 6 \cdot W(2, 5)\) and \(W(3, 4) \approx 11\cdot W(3, 3)\). In three of these four examples one notices that \(W(r, k + 1) \approx (k + 1)\cdot W(r, k)\) for the particular cases \(k + 1 \in \{4, 5, 6\}\). This elicits the question: \emph{How does $W(r, k + 1)$ increase in relation to $W(r, k)$}?\\
\indent However we must caution that, as one also can see, 
$$
W(3, 4) \approx 11\cdot W(3, 3), \: 11 \not = 4,
$$ 
violates any assumption that \(W(r, k + 1) \approx (k + 1)W(r, k)\) strictly is true always! Nevertheless we still were able to find a coarser upper bound on $W(r, k + 1)$ that does have a recurrence property, in that the upper bound on $W(r, k + 1)$ does depend upon the product $(k + 1)W(r, k)$. Moreover although 
\begin{equation}
\frac{W(3, 4)}{W(3, 3)} = \frac{293}{27} = 10.85 \not \approx 4,
\end{equation}
where here \(k + 1 = 4\), we do perceive that 
\begin{eqnarray}
\frac{W(3, 4)}{W(3, 3)}&=&\frac{293}{27} = 10.85 < 4\cdot 3^{4 - 3}\cdot \left(1 + \frac{1}{3}\right)^{4}\\
                       &=&37.92,
\end{eqnarray}
\(k = 3, k + 1 = 4\), which is related directly to the inequality we derive in the next Section.
\section{The  recurrence relationship between van der Waerden Numbers $W(r, k + 1)$ and $W(r, k)$}
Let \(c_{m_{k}} \in [1, k - 1], c_{m_{k} - 1}, c_{m_{k} - 2}, \ldots c_{0, m_{k}} \in [0, k - 1]\) and 
\begin{equation}
c_{m_{k + 1}} \in [1, k], c_{m_{k + 1} - 1}, c_{m_{k + 1} - 2}, \ldots c_{0, m_{k + 1}} \in [0, k],
\end{equation}
be integers and $m_{k}$, $m_{k + 1}$ two positive integer exponents, such that
\begin{eqnarray}
& &W(r, k) = c_{m_{k}}k^{m_{k}} + c_{m_{k} - 1}k^{m_{k} - 1} + \cdots + c_{0, m_{k}} \in [k^{m_{k}}, k^{m_{k} + 1}),\\ 
& &W(r, k + 1) = c_{m_{k + 1}}(k + 1)^{m_{k + 1}} + c_{m_{k + 1} - 1}(k + 1)^{m_{k + 1} - 1} + \cdots + c_{0, m_{k + 1}}\nonumber\\
& &\in [(k + 1)^{m_{k + 1}}, (k + 1)^{m_{k + 1} + 1}).
\end{eqnarray}
In this Section we show that
\begin{equation}
W(r, k + 1) < (k + 1)W(r, k)\cdot k^{m_{k + 1} - m_{k}}(1 + o(1)),
\end{equation}
for large $k$.
\begin{theorem}
For \(k > 2\),
\begin{equation}
W(r, k + 1) < (k + 1)^{m_{k + 1} + 1} \leq (k + 1)W(r, k)\cdot k^{m_{k + 1} - m_{k}}\left(1 + \frac{1}{k}\right)^{m_{k + 1}},
\end{equation}
where 
$$
\left(1 + \frac{1}{k}\right)^{m_{k + 1}} = 1 + o(1),
$$
as $k$ grows large. Moreover
$$
(k + 1)\cdot k^{m_{k + 1} - m_{k}}\left(1 + \frac{1}{k}\right)^{m_{k + 1}} > 1.
$$
\end{theorem}
\begin{proof}
\begin{eqnarray}
\frac{W(r, k + 1)}{W(r, k)}&=     &\frac{c_{m_{k + 1}}(k + 1)^{m_{k + 1}} + c_{m_{k + 1} - 1}(k + 1)^{m_{k + 1} - 1} + \cdots + c_{0, m_{k + 1}}}{c_{m_{k}}k^{m_{k}} + c_{m_{k} - 1}k^{m_{k} - 1} + \cdots + c_{0, m_{k}}}\nonumber\\
                           &<     &\frac{(k + 1)^{m_{k + 1} + 1}}{c_{m_{k}}k^{m_{k}} + c_{m_{k} - 1}k^{m_{k} - 1} + \cdots + c_{0, m_{k}}}\\
                           &=     &\frac{(k + 1)^{m_{k + 1} + 1}}{k^{m_{k}}\left(c_{m_{k}} + \frac{c_{m_{k} - 1}}{k} + \frac{c_{m_{k} - 2}}{k^{2}} + \cdots + \frac{c_{0, m_{k}}}{k^{m_{k}}}\right)}\nonumber\\
                           &=     &\frac{k^{m_{k + 1} - m_{k}}(k + 1)\left(1 + \frac{1}{k}\right)^{m_{k + 1}}}{\left(c_{m_{k}} + \frac{c_{m_{k} - 1}}{k} + \frac{c_{m_{k} - 2}}{k^{2}} + \cdots + \frac{c_{0, m_{k}}}{k^{m_{k}}}\right)}\\
                           &\leq  &\frac{k^{m_{k + 1} - m_{k}}(k + 1)\left(1 + \frac{1}{k}\right)^{m_{k + 1}}}{1},
\end{eqnarray}
since
\begin{equation}
1 \leq \left(c_{m_{k}} + \frac{c_{m_{k} - 1}}{k} + \frac{c_{m_{k} - 2}}{k^{2}} + \cdots + \frac{c_{0, m_{k}}}{k^{m_{k}}}\right), 
\end{equation}
implies Eqns. (25)--(26). One derives the inequality that leads to the right hand side of Eqn. (24) first from the fact that \(W(r, k + 1) < (k + 1)^{m_{k + 1} + 1}\) (See Eqn. (21)), by dividing both $W(r, k + 1)$ and $(k + 1)^{m_{k + 1} + 1}$ by $W(r, k)$, then finally by substituting the radix polynomials for $W(r, k + 1)$ and $W(r, k)$ respectively, that are in Eqns. (20)--(21).  Hence the argument in Eqns. (24)--(26) breaks down to
\begin{equation}
\frac{W(r, k + 1)}{W(r, k)} < \frac{(k + 1)^{m_{k + 1} + 1}}{W(r, k)} \leq k^{m_{k + 1} - m_{k}}(k + 1)\left(1 + \frac{1}{k}\right)^{m_{k + 1}},
\end{equation}
from which we derive, multiplying Eqn. (28) through by $W(r, k)$,
\begin{equation}
W(r, k + 1) < (k + 1)^{m_{k + 1} + 1} \leq (k + 1)W(r, k)\cdot k^{m_{k + 1} - m_{k}}\left(1 + \frac{1}{k}\right)^{m_{k + 1}}.
\end{equation}
Finally since \(|1/k| < 1\) is true for all \(k > 2\) we can expand 
\begin{equation}
\left(1 + \frac{1}{k}\right)^{m_{k + 1}},
\end{equation}
by the Binomial theorem then take limits as \(k \rightarrow \infty\), as 
\begin{eqnarray}
\lim_{k \rightarrow \infty}\left(1 + \frac{1}{k}\right)^{m_{k + 1}}&=&1 + \lim_{k \rightarrow \infty} {m_{k + 1} \choose 1} \frac{1}{k} + \lim_{k \rightarrow \infty}{m_{k + 1} \choose 2}\frac{1}{k^{2}} + \cdots \nonumber\\
                                                                   &=&1 + 0 + 0 + \cdots.
\end{eqnarray}
Therefore we derive as an asymptotic result~\cite{Erdelyi, Knopp, RosenSed}, 
\begin{equation}
\left(1 + \frac{1}{k}\right)^{m_{k + 1}} = (1 + o(1))^{m_{k + 1}} = 1 + o(1),
\end{equation}
which means in Eqn. (29),
\begin{equation}
W(r, k + 1) < (k + 1)^{m_{k + 1} + 1} \leq (k + 1)W(r, k)\cdot k^{m_{k + 1} - m_{k}}(1 + o(1)),
\end{equation}
as $k$ grows large. Finally we must have
$$
(k + 1)\cdot k^{m_{k + 1} - m_{k}}\left(1 + \frac{1}{k}\right)^{m_{k + 1}} > 1,
$$
because in Eqn. (23) and in general even, \(W(r, k + 1) > W(r, k)\).
\end{proof}
The Theorem shows that the rational numbers 
\begin{equation}
\frac{W(r, k + 1)}{W(r, k)},
\end{equation}
are bounded above by
\begin{equation}
(k + 1)k^{m_{k + 1} - m_{k}}\left(1 + \frac{1}{k}\right)^{m_{k + 1}},
\end{equation}
and as $k$ grows large,
$$
\frac{W(r, k + 1)}{W(r, k)} < (k + 1)k^{m_{k + 1} - m_{k}}(1 + o(1)),
$$
where the positive integer exponents $m_{k + 1}$, $m_{k}$, already were described in Eqns. (19)--(21) at the beginning of this Section. Furthermore Theorem 3.1 explains the intriguing computational results we had obtained in Section 2. That is, since in each case we can derive the exponents \(m_{k}, m_{k + 1}\) as needed from Eqns. (5)--(9),
\begin{eqnarray}
\frac{W(2, 4)}{W(2, 3)}&=&\frac{35}{9} = 3.88 \Longrightarrow W(2, 4) = 35 \approx 3.88W(2, 3)\\
                       &<&4W(2, 3)3^{2 - 2}\left(1 + \frac{1}{3}\right)^{2}\nonumber\\
                       &=&64.00,\nonumber \\
\frac{W(2, 5)}{W(2, 4)}&=&\frac{178}{35} = 5.08 \Longrightarrow W(2, 5) = 178 \approx 5.08W(2, 4)\\
                       &<&5W(2, 4)4^{3 - 2}\left( 1 + \frac{1}{4}\right)^{3}\nonumber\\
                       &=&1367.18,\nonumber\\
\frac{W(2, 6)}{W(2, 5)}&=&\frac{1132}{178} = 6.35 \Longrightarrow W(2, 6) = 1132 \approx 6.35W(2, 5)\\
                       &<&6W(2, 5)5^{3 - 3}\left(1 + \frac{1}{5}\right)^{3}\nonumber\\
                       &=&1845.50,\nonumber\\
\frac{W(3, 4)}{W(3, 3)}&=&\frac{293}{27} = 10.85 \Longrightarrow 4W(3, 3)\\
                       &<&W(3, 4) = 293 \approx 10.85W(3, 3)\nonumber\\
                       &<&4W(3, 3)3^{4 - 3}\left( 1 + \frac{1}{3}\right)^{4} = 1023.98.
\end{eqnarray}
Finally using previously known results~\cite{Rabung and Lotts, Betts3}, Theorem 3.1 places bounds on the currently unknown van der Waerden number $W(2, 7)$ as
\begin{eqnarray}
2^{11}&<              &3703 < W(2, 7) < 7W(2, 6)6^{m_{k + 1} - 3}\left(1 + \frac{1}{6}\right)^{m_{k + 1}}\\
      &\Longrightarrow&W(2, 7) \in \left(3703, 7924\cdot 6^{m_{k + 1} - 3}\left(1 + \frac{1}{6}\right)^{m_{k + 1}}\right)\nonumber\\
      &\subset        &\mathbb{R},
\end{eqnarray}
for some positive integer exponent $m_{k + 1}$ and for some nonnegative integers \(c_{m_{k + 1}} \in [1, 6]\), \(c_{m_{k + 1} - 1}, c_{m_{k + 1} - 2}, \ldots, c_{0, m_{k + 1}} \in [0, 6]\)~\cite{Betts3, Betts2}, such that
\begin{equation}
W(2, 7) = c_{m_{k + 1}}7^{m_{k + 1}} + c_{m_{k + 1} - 1}7^{m_{k + 1} - 1} + \cdots + c_{0, m_{k + 1}} \in [7^{m_{k + 1}}, 7^{m_{k + 1} + 1}).
\end{equation}
\subsection{A Corollary based on Theorem 3.1}
Representing $W(r, k + 1)$ as a radix polynomial as was done for Theorem 3.1, reveals that every van der Waerden number $W(r, k)$ lies always between the integer powers $k^{m}$ and $k^{m + 1}$ when the radix is $k$, and that setting \(\delta(r, k) = \log_{k}W(r, k)\) we have that \(\delta(r, k) \in [m, m + 1)\). All this leads us to the fact that there exists a recurrence formula for $W(r, k + 1)$ as well as for $W(r + 1, k)$ (See Section 4). Since with the radix polynomial representation for $W(r, k + 1)$ we know that there exists always a positive real number $\delta(r, k + 1)$, such that 
\begin{equation}
W(r, k + 1) = (k + 1)^{\delta(r, k + 1)},
\end{equation}
where 
\begin{eqnarray}
& &\delta(r, k + 1) = \log_{k + 1} W(r, k + 1),\\
& &m_{k + 1} \leq \delta(r, k + 1) < m_{k + 1} + 1, 
\end{eqnarray}
Theorem 3.1 motivates us to pose the following Corollary, without proof, since the proof is simple not intricate, and straightforward enough to be proved by the reader.
\begin{corollary}
Let \(\delta(r, k + 1) = \log_{k + 1}W(r, k + 1)\) and \(m_{k + 1} = \lfloor \delta(r, k + 1) \rfloor\). Define a function
\begin{equation}
c: \mathbb{N} \times \mathbb{R} \rightarrow \mathbb{R},
\end{equation}
as
$$
c(k, \delta(r, k + 1)) = (k + 1)^{\delta(r, k + 1)} - (k + 1)W(r, k)k^{m_{k + 1} - m_{k}},
$$
where
\begin{equation}
|c(k, \delta(r, k + 1))| = \left|(k + 1)^{\delta(r, k + 1)} - (k + 1)W(r, k)k^{m_{k + 1} - m_{k}}\right|.
\end{equation}
Then
\begin{equation}
W(r, k + 1) = (k + 1)W(r, k)k^{m_{k + 1} - m_{k}} + c(k, \delta(r, k + 1)).
\end{equation}
\end{corollary}
The fact is that by using Eqn. (44), if we define $c(k, \delta(r, k + 1))$ always as 
\begin{equation}
c(k, \delta(r, k + 1)) = (k + 1)^{\delta(r, k + 1)} - (k + 1)W(r, k)k^{m_{k + 1} - m_{k}},
\end{equation}
then there exists some value $c(k, \delta(r, k + 1))$ always, such that the recurrence relation 
$$
W(r, k + 1) = (k + 1)W(r, k)k^{m_{k + 1} - m_{k}} + c(k, \delta(r, k + 1))
$$
exists. The goal then would be, instead of trying to find outright the unknown value of $W(r, k + 1)$ when one already knows $W(r, k)$, $r$, and $k$, either to find or to estimate the positive real exponent $\delta(r, k + 1)$ for which \(W(r, k + 1) = (k + 1)^{\delta(r, k + 1)}\) and the interval 
\begin{equation}
[m_{k + 1}, m_{k + 1} + 1)
\end{equation}
on $\mathbb{R}$, in which this exponent lies. It follows automatically that \(m_{k + 1} = \lfloor \delta(r, k + 1) \rfloor\).
\subsection{Examples of Values for $c(k, \delta(r, k + 1))$}
The examples are for known cases but the values also can be found whenever the value for some unknown $W(r, k + 1)$ is found or as an alternative, when $\delta(r, k + 1)$ is found or estimated first, while $W(r, k)$, $k$, already are known.\\
\indent With \(W(2, 3) = 9, W(2, 4) = 35\), \(m_{3} = m_{4} = 2\), \(\delta(2, 4) = 2.5648\),
\begin{eqnarray}
c(3, \delta(2, 4), W(2, 3))&=&4^{2.5648} - 4W(2, 3)\cdot 3^{2 - 2}\\
                  &=&-0.9923 \approx -1.
\end{eqnarray}
With \(W(2, 4) = 35, W(2, 5) = 178\), \(m_{4} = 2, m_{5} = 3\), \(\delta(2, 5) = 3.2205\),
\begin{eqnarray}
c(4, \delta(2, 5), W(2, 4))&=      &5^{3.2205} - 5W(2, 4)\cdot 4^{1}\\
                  &\approx&-522.
\end{eqnarray}
\section{A Recurrence relation for the van der Waerden Number $W(r + 1, k)$}
There exists also a recurrence formula for $W(r + 1, k)$. Let \(b_{n_{r}} \in [1, r - 1], b_{n_{r} - 1}, b_{n_{r} - 2}, \ldots b_{0, n_{r}} \in [0, r - 1]\) and 
\begin{equation}
b_{n_{r + 1}} \in [1, r], b_{n_{r + 1} - 1}, b_{n_{r + 1} - 2}, \ldots b_{0, n_{r + 1}} \in [0, r],
\end{equation}
be integers and $n_{r}$, $n_{r + 1}$ two positive integer exponents, such that
\begin{equation}
W(r, k) = b_{n_{r}}r^{n_{r}} + b_{n_{r} - 1}r^{n_{r} - 1} + \cdots + b_{0, n_{r}} \in [r^{n_{r}}, r^{n_{r} + 1}),
\end{equation} 
and
\begin{eqnarray}
& &W(r + 1, k) = b_{n_{r + 1}}(r + 1)^{n_{r + 1}} + b_{n_{r + 1} - 1}(r + 1)^{n_{r + 1} - 1} + \cdots + b_{0, n_{r + 1}}\nonumber\\
& &\nonumber\\
& &\in [(r + 1)^{n_{r + 1}}, (r + 1)^{n_{r + 1} + 1}).
\end{eqnarray}
Then just as we were able to prove Eqn. (29) and Eqn. (33), in a manner similar to how we proved Theorem 3.1, we also are able to prove as a corollary to this Theorem, the two inequalities 
\begin{equation}
W(r + 1, k) < (r + 1)^{n_{r + 1} + 1} \leq (r + 1)W(r, k)\cdot r^{n_{r + 1} - n_{r}}\left(1 + \frac{1}{r}\right)^{n_{r + 1}},
\end{equation}
and also
\begin{equation}
W(r + 1, k) < (r + 1)^{n_{r + 1} + 1} \leq (r + 1)W(r, k)\cdot r^{n_{r + 1} - n_{r}}(1 + o(1)),
\end{equation}
when $r$ grows large.\\
\indent Since with the radix polynomial representation for $W(r + 1, k)$ we know that there exists always a real number $\delta(r + 1, k)$, such that 
\begin{equation}
W(r + 1, k) = (r + 1)^{\delta(r + 1, k)},
\end{equation}
where 
\begin{eqnarray}
& &\delta(r + 1, k) = \log_{r + 1} W(r + 1, k),\\
& &n_{r + 1} \leq \delta(r + 1, k) < n_{r + 1} + 1, 
\end{eqnarray}
this motivates us to pose a second Corollary without proof, as it is straightforward enough to be proved by the reader.
\begin{corollary}
Let \(\delta(r + 1, k) = \log_{r + 1}W(r + 1, k)\) and \(n_{r + 1} = \lfloor \delta(r + 1, k) \rfloor\). Define a function
\begin{equation}
b: \mathbb{N} \times \mathbb{R} \rightarrow \mathbb{R},
\end{equation}
as
$$
b(r, \delta(r + 1, k)) = (r + 1)^{\delta(r + 1, k)} - (r + 1)W(r, k)r^{n_{r + 1} - n_{r}},
$$
where 
\begin{equation}
|b(r, \delta(r + 1, k))| = \left|(r + 1)^{\delta(r + 1, k)} - (r + 1)W(r, k)r^{n_{r + 1} - n_{r}}\right|.
\end{equation}
Then
\begin{equation}
W(r + 1, k) = (r + 1)W(r, k)r^{n_{r + 1} - n_{r}} + b(r, \delta(r + 1, k)).
\end{equation}
\end{corollary}
Indeed when we define $b(r, \delta(r + 1, k))$ as 
\begin{equation}
b(r, \delta(r + 1, k)) = (r + 1)^{\delta(r + 1, k)} - (r + 1)W(r, k)r^{n_{r + 1} - n_{r}},
\end{equation}
there is some $b(r, \delta(r + 1, k))$ always whether this value is positive or negative, such that Eqn. (66) holds as a recurrence relation.
In contrast to trying to find or to estimate $W(r + 1, k)$, outright when one knows already what $W(r, k)$, $r$, and $k$ are, the easier goal then could be either to find or to estimate the value of the positive real exponent \(\delta(r + 1, k) \in [n_{r + 1}, n_{r + 1} + 1)\) for which \(W(r + 1, k) = (r + 1)^{\delta(r + 1, k)}\), and the interval 
\begin{equation}
[n_{r + 1}, n_{r + 1} + 1),
\end{equation}
in which $\delta(r + 1, k)$ lies and for which
\begin{equation}
W(r + 1, k) = (r + 1)^{\delta(r + 1, k)},
\end{equation}
where
\begin{equation}
n_{r + 1} = \lfloor \delta(r + 1, k) \rfloor.
\end{equation}
\section{Significance of our results: How $W(r, k + 1)$ and $W(r + 1, k)$ are bounded on $\mathbb{R}$ by the recurrence}
Some peer reviewers and journal editors have expressed the opinion that these results simply are insignificant because they shed no new light on van der Waerden numbers. That of course is their prerogative. Although in the world today power and tradition usually have more influence than does the truth about a matter, still we beg humbly to differ from such claims that these results~\cite{Betts1},~\cite{Betts2},~\cite{Betts3}, have no significance for van der Waerden numbers. \\
\indent We ask what ought to be more significant, proving deep or pure mathematical results for unknown van der Waerden numbers, finding lower bounds on particular unknown van der Waerden numbers like $W(2, 7)$ and $W(2, 8)$, or actually finding the van der Waerden numbers themselves? With all due respect to the anonymous referees and journal editors mentioned previously who admittedly do have a vital role to play in the publishing of research, our results do have \emph{considerable} significance for numerical computing tasks developed to find unknown van der Waerden numbers $W(r, k + 1)$ and $W(r + 1, k)$ and the unknown exponents $m_{k + 1}$, $n_{r + 1}$, when $W(r, k)$, $k$, $m_{k}$, $r$, $n_{r}$ are known already. Since expanding the van der Waerden numbers $W(r, k)$, $W(r, k + 1)$ and $W(r + 1, k)$ by their radix polynomial representations can lead to the finding of such numerical and computational tasks in the future, this hardly can be called insignificant!\\
\indent Based upon the results in Section Three and Section Four we offer two Conjectures. 
\begin{conjecture}
As $k$ grows large, \(k > r\),
\begin{equation}
\frac{W(r, k + 1)}{W(r, k)} < k^{m_{k + 1} - m_{m}}(k + 1).
\end{equation} 
\end{conjecture}
\begin{conjecture}
As $r$ grows large, \(r > k\),
\begin{equation}
\frac{W(r + 1, k)}{W(r, k)} < r^{n_{r + 1} - n_{r}}(r + 1).
\end{equation}
\end{conjecture} 
Assume that \(m_{k + 1} - m_{k} \in [0, 1]\) in Theorem 3.1 and that \(n_{r + 1} - n_{r} \in [0, 1]\) in Eqn.(60) both hold. Then the two conjectures become
\begin{conjecture}
As $k$ grows large, \(k > r\),
\begin{equation}
\frac{W(r, k + 1)}{W(r, k)} < k(k + 1).
\end{equation} 
\end{conjecture}
\begin{conjecture}
As $r$ grows large, \(r > k\),
\begin{equation}
\frac{W(r + 1, k)}{W(r, k)} < r(r + 1).
\end{equation}
\end{conjecture} 
Furthermore the two recurrences we have derived in Section Three and in Section Four when $k$ and $r$ are large, respectively, namely,
\begin{eqnarray}
W(r, k + 1)&<&(k + 1)k^{m_{k + 1} - m_{k}}W(r, k)(1 + o(1)),\\
W(r + 1, k)&<&(r + 1)r^{n_{r + 1} - n_{r}}W(r, k)(1 + o(1)),
\end{eqnarray}
are not in the least insignificant. They indicate how the two van der Waerden numbers $W(r, k + 1)$, $W(r + 1, k)$ even when these are unknown, are bounded on the real line, since
\begin{eqnarray}
(k + 1)^{m_{k + 1}}&\leq&W(r, k + 1) < (k + 1)k^{m_{k + 1} - m_{k}}W(r, k)(1 + o(1)),\\
(r + 1)^{n_{r + 1}}&\leq&W(r + 1, k) < (r + 1)r^{n_{r + 1} - n_{r}}W(r, k)(1 + o(1)).
\end{eqnarray}
Such facts, which we choose to call information or facts we have \emph{a priori} about the two unknown van der Waerden numbers $W(r, k + 1)$, $W(r + 1, k)$, cannot be insignificant since for large $k$ and $r$ respectively and from the two inequalities in Eqns. (75)--(76) we can derive also the two inequalities
\begin{eqnarray}
\log_{k}W(r, k + 1) - m_{k + 1}&\leq&\log_{k}((k + 1)W(r, k)) - m_{k},\\
\log_{r}W(r + 1, k) - n_{r + 1}&\leq&\log_{r}((r + 1)W(r, k)) - n_{r}.
\end{eqnarray}
This means in the first case for $k$ large enough (i.e., so that $\left(1 + \frac{1}{k}\right)^{m_{k + 1}}$ is $1 + o(1)$), by computer and even through trial and error if necessary, we either can find estimates or approximations for the unknown values $\log_{k}W(r, k + 1)$ and $m_{k + 1}$ by bounding each of our estimated or approximated differences for $\log_{k}W(r, k + 1) - m_{k + 1}$ above by 
\begin{equation}
\log_{k}((k + 1)W(r, k)) - m_{k}.
\end{equation}
In the second case for $r$ large enough (i.e., so that $\left(1 + \frac{1}{r}\right)^{n_{r + 1}}$ is $1 + o(1)$), we also through trial and error if necessary, either can find estimates or approximations for the unknown values $\log_{r}W(r + 1, k)$ and $n_{r + 1}$ by bounding each of the estimated or approximated differences for $\log_{r}W(r + 1, k) - n_{r + 1}$ above by 
\begin{equation}
\log_{r}((r + 1)W(r, k)) - n_{r}.
\end{equation} 
Our approach also poses a question we feel very strongly is important in reference to van der Waerden numbers: What are the asymptotic behaviors of the two differences 
\begin{equation}
m_{k + 1} - m_{k},
\end{equation}
and
\begin{equation}
n_{r + 1} - n_{r}?
\end{equation}
Thus if one insists still that all this is insignificant, certainly when it comes to finding through numerical methods and algorithms the estimates or approximations for the unknown numbers $W(r, k + 1)$, $m_{k + 1}$, and $W(r + 1, k)$, $n_{r + 1}$ whenever $W(r, k)$, $r$, $k$, $m_{k}$, $n_{r}$ already are 
known, then one only can ask as did Pontius Pilate, \emph{Quid est veritas}?

\pagebreak

\end{document}